\documentclass[11pt]{amsart}
\usepackage{amsfonts, amssymb, amsmath}
\usepackage{mathrsfs}
\usepackage{tikz-cd}
\usepackage{url}
\usepackage{bbm}
\usepackage[hidelinks]{hyperref}
\usepackage{adjustbox}
\newcommand{\spa}{\medskip}

\usepackage[left=22.5mm,right=22.5mm,top=38mm,bottom=34mm]{geometry}

\newcommand{\F}{\mathbb{F}}
\newcommand{\Fp}{\mathbb{F}_p}

\newcommand{\Z}{\mathbb{Z}}

\newcommand{\kbar}{{\bar{k}}}

\newcommand{\Zp}{{\mathbb{Z}_{p}}}
\newcommand{\Qp}{{\mathbb{Q}_{p}}}

\newcommand{\Gm}{\mathbb{G}_m}

\newcommand{\sym}{{\mathrm{sym}}}

\newcommand{\calA}{\mathcal{A}}

\newcommand{\calE}{\mathcal{E}}
\newcommand{\calF}{\mathcal{F}}
\newcommand{\calG}{\mathcal{G}}
\newcommand{\calH}{\mathcal{H}}
\newcommand{\calL}{\mathcal{L}}
\newcommand{\calO}{\mathcal{O}}
\newcommand{\calP}{\mathcal{P}}
\newcommand{\calM}{\mathcal{M}}
\newcommand{\calN}{\mathcal{N}}

\newcommand{\calY}{\mathcal{Y}}

\newcommand{\Fcrys}{\mathbf{F\textrm{-}Crys}}

\newcommand{\myim}{\mathrm{im}}

\DeclareMathOperator{\Spec}{Spec}

\DeclareMathOperator{\Hom}{Hom}

\DeclareMathOperator{\End}{End}
\DeclareMathOperator{\Aut}{Aut}

\DeclareMathOperator{\id}{id}

\DeclareMathOperator{\rk}{rk}

\newcommand{\HHom}{\mathcal{H}om}
\newcommand{\et}{\mathrm{\acute{e}t}}

\newcommand{\crys}{{\mathrm{crys}}}
\newcommand{\Pic}{\mathrm{Pic}}
\newcommand{\Br}{\mathrm{Br}}
\newcommand{\NS}{\mathrm{NS}}

\newcommand{\fppf}{\mathrm{fppf}}

\newcommand{\iso}{\xrightarrow{\sim}}

\newcommand{\Mor}{\mathrm{Mor}}
\newcommand{\ks}{{k_s}}
\newcommand{\ki}{{k_i}}

\newcommand{\mupn}{{\mu_{p^n}}}

\newcommand{\Tp}{\mathrm{T}_p}
\newcommand{\Vp}{\mathrm{V}_p}
\newcommand{\Tl}{\mathrm{T_\ell}}

\begin{document}

\newtheorem{theo}[subsection]{Theorem}
\newtheorem*{theo*}{Theorem}
\newtheorem{conj}[subsection]{Conjecture}

\newtheorem{prop}[subsection]{Proposition}
\newtheorem{lemm}[subsection]{Lemma}
\newtheorem*{lemm*}{Lemma}
\newtheorem{coro}[subsection]{Corollary}

\theoremstyle{definition}
\newtheorem{defi}[subsection]{Definition}
\newtheorem*{defi*}{Definition}
\newtheorem{hypo}[subsection]{Hypothesis}
\newtheorem{rema}[subsection]{Remark}
\newtheorem{exam}[subsection]{Example}
\newtheorem{nota}[subsection]{Notation}
\newtheorem{cons}[subsection]{Construction}
\numberwithin{equation}{subsection}

\setcounter{tocdepth}{1}

\title{Boundedness of the $p$-primary torsion of the Brauer group of an abelian variety}
\date{\today}
\makeatletter
	\@namedef{subjclassname@2020}{%
	\textup{2020} Mathematics Subject Classification}
\makeatother

\author{Marco D'Addezio}

\address{Institut de Mathématiques de Jussieu-Paris Rive Gauche, SU - 4 place Jussieu, Case 247, 75005 Paris}
\email{daddezio@imj-prg.fr}
\begin{abstract}
	We prove that the $p^\infty$-torsion of the transcendental Brauer group of an abelian variety over a finitely generated field of characteristic $p>0$ is bounded. This answers a (variant of a) question asked by Skorobogatov and Zarhin for abelian varieties. To do this, we prove a “flat Tate conjecture” for divisors. In the text, we also study other geometric Galois-invariant $p^\infty$-torsion classes of the Brauer group which are not in the transcendental Brauer group. These classes, in contrast with our main theorem, can be infinitely $p$-divisible. We explain how the existence of these $p$-divisible towers is naturally related to the failure of surjectivity of specialisation morphisms of Néron--Severi groups in characteristic $p$.
\end{abstract}
\maketitle

	\tableofcontents
\section{Introduction}

In this article we want to study problems related to the finiteness of the $p$-primary torsion of the Brauer group of abelian varieties in positive characteristic $p$. If $k$ is a finite field and $A$ is an abelian variety over $k$, it is well-known that the Brauer group of $A$, defined as $\Br(A):=H^2_\et(A,\Gm)$, is a finite group, \cite[Prop. 4.3]{Tat94}. The main input for this result is the Tate conjecture for divisors, proved by Tate in \cite{Tat66}. If $k$ is replaced by a finitely generated field extension of $\Fp$ one can not expect anymore $\Br(A)$ to be finite (see \cite[§1]{SZ08}). On the other hand, if $\Br(A_{\ks})^{k}$ is the \textit{transcendental Brauer group} of $A$, namely the image of $\Br(A)\to \Br(A_{\ks})$ where $\ks$ is a separable closure of $k$, the group $\Br(A_{\ks})^{k}[\tfrac{1}{p}]$ is finite by \cite[Thm. 16.2.3]{CS21}. In \cite[Ques. 1]{SZ08}, Skorobogatov and Zarhin asked whether the $p$-primary torsion of $\Br(A_{\ks})^{k}$ is finite as well. This question has a negative answer already for abelian surfaces, as we show in Proposition \ref{counterexample:p}. Nonetheless, we prove the following alternative finiteness result. Write $\kbar$ for an algebraic closure of $k_s$.

\begin{theo}[Theorem \ref{fini-Brau:t}]\label{i-main:t}Let $A$ be an abelian variety over a finitely generated field $k$ of characteristic $p>0$. The transcendental Brauer group $\Br(A_{\ks})^{k}$ is a direct sum of a finite group and a finite exponent $p$-group. In addition, if the Witt vector cohomology group $H^2(A_{\kbar},W\calO_{A_\kbar})$ is a finite $W(\kbar)$-module, then $\Br(A_{\ks})^{k}$ is finite.
\end{theo}

The condition on $H^2(A_{\kbar},W\calO_{A_\kbar})$ is necessary to remove the “supersingular pathologies” as the one of our counterexample and it is satisfied, for example, when the $p$-rank of $A$ is $g$ or $g-1$, where $g$ is the dimension of $A$ (see \cite[Cor. 6.3.16]{Ill83}). Note that if the \textit{formal Brauer group} of $A_\kbar$, denoted by $\hat{\Br}(A_\kbar)$, is a formal Lie group, then by \cite[Cor. II.4.4]{AM77} the cohomology group $H^2(A_{\kbar},W\calO_{A_\kbar})$ is a finite  $W(\kbar)$-module if and only if $\hat{\Br}(A_\kbar)$ has finite height. Note also that the formal Brauer group of abelian surfaces is always a formal Lie group by [\textit{ibid.}, Cor. II.2.12]. As a consequence of Theorem \ref{i-main:t}, we deduce that the subgroup of Galois-fixed points of $\Br(A_{\ks})$, denoted by $\Br(A_{\ks})^{\Gamma_k}$, has finite exponent (Corollary \ref{galo-fixe:c}). This is a variant of \cite[Ques. 2]{SZ08} for abelian varieties. 

\spa

In this article, we also study the Galois-fixed points of $\Br(A_{\kbar})$. Ulmer in \cite[§7.3.1]{Ulm14} conjectured that $\Tp(\Br(A_\kbar))^{\Gamma_k}=0$ where $\Tp(\Br(A_\kbar))$ is the $p$-adic Tate module of $\Br(A_{\kbar})$. Even in this case, we provide a counterexample to this conjecture. We use the following result.

\begin{prop}[Proposition \ref{counter:p}]\label{i-counter:p}
	Let $B$ be an abelian variety over a finitely generated field $k$ of characteristic $p>0$. Write $A$ for $B\times_k B$ and $\Tp(\Br(A_\kbar))$ for the $p$-adic Tate module of $\Br(A_\kbar)$. There is a natural exact sequence		$$0\to \Hom(B,B^\vee)_\Zp\to \Hom(B_\kbar[p^\infty],B_\kbar^\vee[p^\infty])^{\Gamma_k}\to \Tp(\Br(A_\kbar))^{\Gamma_k},$$
	
	where $\Hom(B,B^\vee)$ denotes the group of homomorphisms $B\to B^\vee$ as abelian varieties over $k$.

\end{prop}

	The proposition implies, for example, that when $\End(B)=\Z$ the $\Gamma_k$-module $\Tp(\mathrm{Br}(A_{\bar{k}}))$ admits non-zero Galois-fixed points (Corollary \ref{count:c}). In this case, $\Br(A_{\kbar})^{\Gamma_k}$ has infinite exponent since $$\Tp(\Br(A_\kbar)^{\Gamma_k})=\Tp(\Br(A_\kbar))^{\Gamma_k}.$$ Note that if we replace $\Tp(\mathrm{Br}(A_{\bar{k}}))$ with the $\ell$-adic Tate module $\Tl(\mathrm{Br}(A_{\bar{k}}))$, where $\ell$ is a prime different from $p$, then $\Tl(\mathrm{Br}(A_{\ks}))=\Tl(\mathrm{Br}(A_{\bar{k}}))$ has no non-trivial Galois-fixed points. 
	
	\spa

	These “exceptional classes” in $\Tp(\Br(A_\kbar))^{\Gamma_k}$ are naturally related to specialisation morphisms of Néron--Severi groups. We recall the following theorem, which was proved in \cite[Thm. 5.2]{And96} in characteristic $0$ (see also \cite{MP12}) and in \cite{Amb18} and \cite{Chr18} in positive characteristic. 
	
	\begin{theo}[Andr\'e,  Ambrosi, Christensen]\label{i-spec:t}
Let $K$ be an algebraically closed field which is not an algebraic extension of a finite field, $X$ a finite type $K$-scheme, and $\calY\to X$ a smooth proper morphism. For every geometric point $\bar{\eta}$ of $X$ there is an $x\in X(K)$ such that $\rk_\Z(\NS(\calY_{\bar{\eta}}))=\rk_\Z(\NS(\calY_x))$\footnote{If $R$ is a domain with fraction field $K$ and $M$ is an $R$-module, we write $\rk_R(M)$ for the dimension of $M\otimes_R K$ as a $K$-vector space.}.
	\end{theo}
	
	As it is well-known, the theorem is false when $K=\bar{\F}_p$ (see \cite[Rmk. 1.12]{MP12}). What we prove is that, in the known counterexamples, the elements in $\Tp(\Br(A_\kbar))^{\Gamma_k}$ explain the failure of Theorem \ref{i-spec:t}. More precisely, we prove the following result.

\begin{theo}[Theorem \ref{NS:t}]\label{i-NS:t}
	Let $X$ be a connected normal scheme of finite type over $\Fp$ with generic point $\eta=\Spec(k)$ and let $f:\calA\to X$ be an abelian scheme over $X$ with constant Newton polygon\footnote{With this we mean that for every algebraically closed field $\Omega$ and every $\bar{x}\in X(\Omega)$, the Newton polygons of the fibres $\calA_{\bar{x}}$ are all equal. Note that in this case it is enough to check $\bar{\mathbb{F}}_p$-points.}. For every closed point $x=\Spec(\kappa)$ of $X$ we have $$\rk_\Z(\NS(\calA_{\bar{x}})^{\Gamma_\kappa})-\rk_\Z(\NS(\calA_{\bar{\eta}})^{\Gamma_k})\geq \rk_\Zp ( \Tp(\Br(\calA_{\bar{\eta}}))^{\Gamma_k}).$$
\end{theo}
Note that in the inequality the left term is “motivic”, while the right term comes from some $p$-adic object which, as far as we know, has no $\ell$-adic analogue. Note also that $\Tp(\Br(\calA_{\bar{x}}))^{\Gamma_\kappa}=0$ by Corollary \ref{galo-fixe:c} since $\kappa$ is a perfect field.

\spa

To prove Theorem \ref{i-main:t} we use a flat variant of the Tate conjecture. For every $n$, let $H^2_\fppf(A_\kbar,\mupn)^k$ be the image of the extension of scalars morphism $H^2_\fppf(A,\mupn)\to H^2_\fppf(A_\kbar,\mupn)$.

\begin{theo}[Theorem \ref{fppf-tate:t}]\label{i-fppf-tate:t}
After possibly replacing $k$ with a finite separable extension, the cycle class map $$c_1:\NS(A)_{\Zp}\to\varprojlim_n H^2_\fppf(A_\kbar,\mupn)^k$$ becomes an isomorphism\footnote{For us, $\NS(A)$ is the group of $k$-points of the group scheme $\pi_0(\Pic_{A/k})$.}.
\end{theo}

We obtain this result by using the crystalline Tate conjecture for abelian varieties, proved by de Jong in \cite[Thm. 2.6]{deJ98}. The main issue that we have to overcome is the lack of a good comparison between crystalline and fppf cohomology of $\Zp(1)$ over imperfect fields. To avoid this problem, we exploit the fact that we are working with abelian varieties. In this special case, the comparison is constructed using the $p$-divisible group of $A$ (and its dual). 

\spa

The technical issue that we have to solve using the groups $H^2_{\fppf}(A_\kbar,\mupn)^k$ is that it is not clear a priori whether $H^2_{\fppf}(A,\Zp(1))\to \varprojlim_n H^2_{\fppf}(A_\kbar,\mupn)^k$ is surjective. This is done (after inverting $p$) in Proposition \ref{tricky:p}, where we reduce to the case when $A$ is the Jacobian of a curve. This idea was inspired by the proof of \cite[Thm. 2.1]{CS13}.

\subsection{Outline of the article}In §\ref{prel-resu:s} we prove some general results on the cohomology of fppf sheaves. In particular, we prove Corollary \ref{Br-inj:c}, which is a first result on the relation between the Brauer group of a scheme over $\ks$ and $\kbar$. In this section, we also prove in Proposition \ref{Kummer-k-inv:p} the exactness of some fundamental sequences for the groups $H^2_\fppf(X_\kbar,\mupn)^k$. In §\ref{cons-morp:s}, we construct a morphism which relates $H^2_\fppf(A,\Zp(1))$ with $\Hom(A[p^\infty],A^\vee[p^\infty])$ and we prove basic properties of this morphism as Proposition \ref{comm-squa:p} and Proposition \ref{imag:p}. In §\ref{main-resu:s}, we prove the flat variant of the Tate conjecture (Theorem \ref{fppf-tate:t}) and the finiteness result for the transcendental Brauer group (Theorem \ref{fini-Brau:t}). Finally, in §\ref{spec:s}, we look at the relation of our results with the theory of specialisation of Néron--Severi groups. In particular, we prove Theorem \ref{NS:t}.

\subsection{Acknowledgments} I thank Emiliano Ambrosi for the discussions we had during the writing of \cite{AD18} which inspired this article, Matthew Morrow and Kay Rülling for many enlightening conversations about the cohomology of $\Z_p(1)$, and Ofer Gabber, Luc Illusie, Peter Scholze, and Takashi Suzuki for answering some questions on the fppf site. I also thank Jean-Louis Colliot-Thélène, Bruno Kahn, Alexei Skorobogatov, and Takashi Suzuki for very useful comments on a first draft of this article. Finally, I thank the anonymous referees for their careful reading of the article and for the corrections they suggested.

	\spa

The author was funded by the Deutsche Forschungsgemeinschaft (EXC-2046/1, project ID: 390685689 and DA-2534/1-1, project ID: 461915680) and by the Max-Planck Institute for Mathematics.
\section{Notation}

If $k$ is a field, we write $\kbar$ for a fixed algebraic closure of $k$ and $\ks$ (resp. $\ki$) for the separable (resp. purely inseparable) closure of $k$ in $\kbar$. We denote by $\Gamma_k$ the absolute Galois group of $k$. If $x$ is a $k$-point of a scheme, we denote by $\bar{x}$ the induced $\kbar$-point. For an abelian group $M$, we write $\Tp(M)$ for the $p$-adic Tate module of $M$, which is the projective limit $\varprojlim_{n}M[p^n]$, we write $\Vp(M)$ for $\Tp(M)[\tfrac{1}{p}]$, and we write $M^\wedge$ for the $p$-adic completion of $M$. If $M$ is endowed with a $\Gamma_k$-action, we denote by $M^{\Gamma_k}$ the subgroup of fixed points. For a scheme $X$ and an fppf sheaf $\calF$, we denote by $H^{\bullet}(X,\calF)$ the fppf cohomology groups and when $X=\Spec(k)$ we simply write $H^{\bullet}(k,\calF)$. If $f:X\to Y$ is a morphism of schemes, we denote by $R^{\bullet}f_*\calF$ the fppf higher direct image functors over $(\mathbf{Sch}/Y)_\fppf$. Finally, if $X$ is a scheme over $\Fp$, we write $X^{\mathrm{perf}}$ for the projective limit $\varprojlim(...\xrightarrow{F}X\xrightarrow{F}X\xrightarrow{F}X)$ where $F$ is the absolute Frobenius of $X$.
\section{Preliminary results}\label{prel-resu:s}

In this section we start by proving some results that we will use later on. We work over a field $k$ of arbitrary characteristic and we consider a scheme $X$ over $k$ with structural morphism $q$.

\begin{lemm}\label{fppf-exte-scal:l}
Let $\calF$ be a sheaf over $(\mathbf{Sch}/k)_\fppf$ such that $q_*\calF_{X}=\calF$ and suppose that $X$ has a $k$-rational point. The group  $H^0(k,R^1q_*\calF_{X})$ is canonically isomorphic to $H^1(X,\calF_X)/H^1(k,\calF)$. In addition, the natural morphism $H^2(X,\calF_X)\to H^0(k,R^2q_*\calF_{X})$ sits in an exact sequence 
$$0\to K\to H^2(X,\calF_X)\to H^0(k,R^2q_*\calF_{X})\to H^2(k,R^1q_*\calF_{X})$$
where $K$ is an extension of $H^1(k,R^1q_*\calF_{X})$ by $H^2(k,\calF)$. 
\end{lemm}

\begin{proof}
We consider the Leray spectral sequence
$$E^{i,j}_2=H^i(k,R^jq_{*}\calF_X)\Rightarrow H^{i+j}(X,\calF_X).$$ 	The morphisms $E^{i,0}_2=H^i(k,q_*\calF_X)=H^i(k,\calF)\to H^i(X,\calF_X)$ are injective since $X$ admits a $k$-rational point. We deduce that $E^{1,1}_2=E^{1,1}_\infty$ and $E^{2,0}_2=E^{2,0}_\infty$. This implies that the kernel of $H^2(X,\calF_X)\to E^{0,2}_\infty$ is an extension of $E^{1,1}_2$ by $E^{2,0}_2$, as we wanted. The obstruction for the map $H^2(X,\calF_X)\to E^{0,2}_2=H^0(k,R^2q_*\calF_{X})$ to be surjective lies in $E^{2,1}_2=H^2(k,R^1q_*\calF_X)$. This concludes the proof.
\end{proof}

\begin{defi}
	We say that a presheaf $\calF$ on $(\mathbf{Sch}^{\mathrm{qcqs}}/k)_{\fppf}$ is \textit{finitary} if for every inverse system $\{T^{(\ell)}\}_{\ell \in L}$ of quasi-compact quasi-separated $k$-schemes with affine transition maps, the natural morphism 

$$\mathrm{colim}_{\ell\in L}\calF(T^{(\ell)})\to \calF(\lim_{\ell\in L} T^{(\ell)})$$ is an isomorphism.
\end{defi}
\begin{lemm}\label{limi:l}
	Let $G$ be a commutative finite type group scheme over $k$. If $X$ is quasi-compact quasi-separated, then $R^iq_*G_X$ is finitary for $i\geq 0$. In addition, the natural morphism $H^0(k,R^i q_*G_X)\to H^{i}(X_\kbar,G_{X_\kbar})$ is injective.

\end{lemm}
\begin{proof}
 Let ${\calH}^i(q,G_{X})$ be the higher presheaf pushforward of $G_{X}$ on $X$ with respect to $q$. We first want to prove that ${\calH}^i(q,G_{X})$ is finitary for $i\geq 0$. In other words, we want to prove that for every inverse system $\{T^{(\ell)}\}_{\ell \in L}$ of quasi-compact quasi-separated $k$-schemes, the natural morphism $$\mathrm{colim}_{\ell\in L}H^i(X^{(\ell)},G_{X^{(\ell)}})\to H^i(X^{(\infty)},G_{X^{(\infty)}})$$ is an isomorphism, where $X^{(\ell)}:=X\times_k T^{(\ell)}$ and $X^{(\infty)}:=\lim_{\ell\in L}X^{(\ell)}$. By \cite[Tag 01H0]{SP22},  $$H^i(X^{(\ell)},G_{X^{(\ell)}})=\mathrm{colim}_{U_\bullet^{(\ell)}\in \mathrm{HC}(X^{(\ell)})}\check{H}^i(U_\bullet^{(\ell)},G_{U_\bullet^{(\ell)}})$$
 for every $\ell\in L \coprod \{\infty\}$, where $\mathrm{HC}(X^{(\ell)})$ is the category of fppf hypercoverings of $X^{(\ell)}$. Since each $X^{(\ell)}$ is quasi-compact quasi-separated, by [\textit{ibid.}, Tag 021P], we can replace the category $\mathrm{HC}(X^{(\ell)})$ in the colimit with the subcategory $\mathrm{HC}(X^{(\ell)})^{\mathrm{qcqs}}$, consisting of those hypercoverings such that $U^{(\ell)}_n$ is quasi-compact quasi-separated for every $n\geq 0$.  By [\textit{ibid.}, Lem. 01ZM], for $U^{(\infty)}_\bullet\in\mathrm{HC}(X^{(\infty)})^{\mathrm{qcqs}}$ and $n\geq 0$ there exists an $\ell\in L$ and $U_\bullet^{(\ell,n)}\in \mathrm{HC}(X^{(\ell)})^{\mathrm{qcqs}}$ such that $$\mathrm{tr}_n(U_\bullet^{(\ell,n)}\times_{T^{(\ell)}} T^{(\infty)})\simeq \mathrm{tr}_n(U^{(\infty)}_\bullet),$$ where $\mathrm{tr}_n(-)$ denotes the $n$-th truncation of simplicial schemes and $T^{(\infty)}:=\lim_{\ell\in L}T^{(\ell)}$. This implies that $$H^i(X^{(\infty)},G_{X^{(\infty)}})=\mathrm{colim}_{\ell\in L}\mathrm{colim}_{U_\bullet^{(\ell)}\in \mathrm{HC}(X^{(\ell)})}\check{H}^i(U_\bullet^{(\ell)}\times_{T^{(\ell)}} T^{(\infty)},G_{U_\bullet^{(\ell)}\times_{T^{(\ell)}} T^{(\infty)}}).$$

  We are reduced to prove that for every $\ell\in L$ and $U_\bullet^{(\ell)}\in \mathrm{HC}(X^{(\ell)})^{\mathrm{qcqs}}$ we have that $$\check{H}^i(U_\bullet^{(\ell)}\times_{T^{(\ell)}}T^{(\infty)},G_{U_\bullet^{(\ell)}\times_{T^{(\ell)}}T^{(\infty)}})=\mathrm{colim}_{\ell\leq \ell'}\check{H}^i(U_\bullet^{(\ell)}\times_{T^{(\ell)}}T^{(\ell')},G_{U_\bullet^{(\ell)}\times_{T^{(\ell)}}T^{(\ell')}}).$$ Since $G$ is of finite type over $k$, this follows from [\textit{ibid.}, Lem. 01ZM] and the exactness of filtered colimits.
 
\spa

Knowing that ${\calH}^i(q,G_{X})$ is finitary, in order to prove that $R^iq_*G_X$ is finitary as well it is enough to prove that for every finitary presheaf $\calF$ on $(\mathbf{Sch}^{\mathrm{qcqs}}/k)_{\fppf}$, the “partial” sheafification $\calF^+$ (defined as in [\textit{ibid.}, §00W1]) is finitary. Similarly to the previous paragraph, the proof of this fact follows from the observation that each finite quasi-compact quasi-separated fppf covering of $X^{(\infty)}$ descends to a covering of $X^{(\ell)}$ for some $\ell\in L$ and Čech cohomology commutes with filtered colimits ($\check{H}^0$ is enough in this case).

	\spa
	For the second part, we note that for every presheaf $\calF$ on $(\mathbf{Sch}/k)_{\mathrm{fppf}}$ with sheafification $\calF^\sharp$, the natural morphism $$\calF(\Spec(\kbar))\to \calF^\sharp(\Spec(\kbar))$$ is an isomorphism because every fppf covering of $\Spec(\kbar)$ admits a section. This implies that $$H^0(\kbar,R^iq_*G_X)= H^{i}(X_\kbar,G_{X_\kbar}).$$
	
	Thanks to the previous part, we deduce that the composition $$H^0(k,R^i q_*G_X)\hookrightarrow \mathrm{colim}_{{k'/k\  \mathrm{fin.}}}{H}^0(k',R^i q_*G_X)\iso H^0(\kbar,R^iq_*G_X)= H^{i}(X_\kbar,G_{X_\kbar})$$ is injective, where the colimit runs over all finite field extensions of $k$. This ends the proof.\end{proof}

With the previous results we can prove \cite[Prop. 5.6]{Gro68}, which was stated by Grothendieck without a complete proof\footnote{Note that the result is also proven in \cite[Thm. 5.2.5.i]{CS21}, but their proof has a gap since the justification of the fact that $H^0(k,R^2p_*\mathbb{G}_{m,X})\to H^0(\kbar,R^2p_*\mathbb{G}_{m,X})$ is injective is not correct.}.
\begin{coro}\label{Br-inj:c}If $k$ is separably closed and $X$ is a proper $k$-scheme, then there is a natural exact sequence $$0\to H^1(k,\Pic_{X/k})\to\Br(X)\to \Br(X_{\kbar}).$$ In particular, if $\Pic_{X/k}$ is smooth then the natural morphism $\Br(X)\to \Br(X_{\kbar})$ is injective.
\end{coro}

\begin{proof}As in Lemma \ref{fppf-exte-scal:l}, we consider the Leray spectral sequence
	$$E^{i,j}_2=H^i(k,R^jq_{*}\mathbb{G}_{m,X})\Rightarrow H^{i+j}(X,\mathbb{G}_{m,X}).$$ Since $X$ is proper over $k$, by \cite[Tag 0BUG]{SP22} we deduce that $A:=H^0(X,\calO_X)$ is a finite $k$-algebra. This implies that $q_*\Gm$ is represented by a smooth group scheme over $k$. Thanks to \cite[Thm. 11.7]{Gro68}, we deduce that $E^{i,j}_2=0$ for $i>0$ and $j=0$, so that $E^{1,1}_2=E^{1,1}_\infty$. The Leray spectral sequence produces then the exact sequence $$0\to H^1(k,\Pic_{X/k})\to\Br(X)\to H^0(k,R^2q_*\mathbb{G}_{m,X}).$$
	
	To get the first part of the statement it is then enough to apply Lemma \ref{limi:l}. For the second part, we note that when $\Pic_{X/{k}}$ is smooth, thanks to [\textit{ibid.}, Thm. 11.7], the group $H^1(k,\Pic_{X/{k}})$ vanishes.
\end{proof}

\begin{defi}
For a scheme $X$ over $k$ and a prime $p$, we define $H^2(X,\Z_p(1))$ as the projective limit $$\varprojlim_{n} H^2(X,\mupn).$$
\end{defi}
\begin{rema}Note that we are defining $H^2(X,\Z_p(1))$ without taking into account higher inverse limits. Nonetheless, if $k$ is algebraically closed of characteristic $p$ and $X$ is smooth and proper over $k$, then $R^1\varprojlim_n H^1(X,\mupn)=R^1\varprojlim_n \Pic(X)[p^n]=0$ since $\Pic(X)[p^\infty]$ is a direct sum of a $p$-divisible group and a finite group and $R^1\varprojlim_n H^2(X,\mupn)=0$ by 
\cite[Chap. II, Prop. 5.9]{Ill79}.
\end{rema}

\begin{cons}
The Kummer exact sequences for $X$ and $X_{\kbar}$ (for the fppf topology) induce the following commutative diagram with exact rows
	\begin{equation}\label{comp:e}\begin{tikzcd}
			0\arrow[r] &\Pic(X)/p^n \arrow[d,""]\arrow[r] & H^2_{}(X,\mupn)\arrow[d,""]\arrow[r] & \mathrm{Br}(X)[p^n]\arrow[r]\arrow[d,""] & 0\\
			0\arrow[r]& \Pic(X_\kbar)/p^n \arrow[r]& H^2_{}(X_\kbar,\mupn)\arrow[r]& \mathrm{Br}(X_\kbar)[p^n]\arrow[r]& 0.
		\end{tikzcd}
	\end{equation}

We write 	$$C_n(X):=(\Pic(X_\kbar)/p^n)^k\to H^2(X_\kbar,\mupn)^{k}\to (\Br(X_{\kbar})[p^n])^{k}\to 0\to \dots$$
for the complex obtained by taking images of the vertical arrows. Note that a priori $(\Br(X_{\kbar})[p^n])^{k}$ is smaller than $\Br(X_\kbar)^k[p^n]$, where $\Br(X_\kbar)^k:=\myim(\Br(X)\to \Br(X_\kbar))$.

\spa

Since both $R^1\varprojlim_{n}\Pic(X)/p^n$ and $R^1\varprojlim_{n}\Pic(X_\kbar)/p^n$ vanish, we can also consider the following commutative diagram with exact rows
\begin{equation*}\label{diag:e}\begin{tikzcd}
		0\arrow[r] &\Pic(X)^\wedge \arrow[d,""]\arrow[r] & H^2_{}(X,\Zp(1))\arrow[d,""]\arrow[r] & \Tp(\mathrm{Br}(X))\arrow[r]\arrow[d,""] & 0\\
		0\arrow[r]& \Pic(X_\kbar)^\wedge \arrow[r]& H^2_{}(X_\kbar,\Zp(1))\arrow[r]& \Tp(\mathrm{Br}(X_\kbar))\arrow[r]& 0,
	\end{tikzcd}
\end{equation*}

obtained by taking the projective limit of the diagrams (\ref{comp:e}) for various $n$. We denote by $$\hat{C}(X):=
	(\Pic(X_\kbar)^\wedge)^k\to H^2(X_\kbar,\Z_p(1))^{k}\to \Tp(\Br(X_{\kbar}))^{k}\to 0\to \dots
$$
the complex obtained by taking images of the vertical arrows.
\end{cons}

\begin{prop}\label{Kummer-k-inv:p} If $\mathrm{char}(k)=p$ and $A$ is an abelian variety over $k$ such that the morphism $\Pic(A)\to \NS(A_{\kbar})$ is surjective, then the complexes $C_n(A)$ and $\hat{C}(A)$ are acyclic.
\end{prop}
\begin{proof}

If $K_{1,n}$ is the kernel of $H^2_{}(A,\mupn)\to H^2(A_\kbar,\mupn)$ and $K_{2}$ is the kernel of $\Br(A)\to \Br(A_\kbar)$, in order to prove that $C_n(A)$ is acyclic we have to show that $K_{1,n}\to K_{2}[p^n]$ is surjective. Combining Lemma \ref{fppf-exte-scal:l} and Lemma \ref{limi:l}, we deduce the following commutative diagram with exact rows 

	\begin{equation*}\begin{tikzcd}
		0\arrow[r] &\Br(k)[p^n]\arrow[d,""]\arrow[r] &  K_{1,n}\arrow[d,""]\arrow[r] & H^1_{}(k,\Pic_{A/k}[p^n])\arrow[r]\arrow[d,""] & 0\\
		0\arrow[r]& \Br(k) \arrow[r]& K_2\arrow[r]& H^1_{}(k,\Pic_{A/k})\arrow[r]  & 0.
	\end{tikzcd}
\end{equation*}

The morphism of exact sequences factors through the complex $$\Br(k)[p^n]\to K_2[p^n]\to H^1_{}(k,\Pic_{A/k})[p^n]\to 0\to \dots$$ which is acyclic because $\Br(k)$ is $p$-divisible by \cite[Thm. 1.3.7]{CS21}. 
The image of $$H^1(k,\Pic_{A/k}[p^n])\to H^1_{}(k,\Pic_{A/k}^\circ)$$ is $H^1_{}(k,\Pic_{A/k}^\circ)[p^n]$, thus we are reduced to prove that $$H^1_{}(k,\Pic_{A/k}^\circ)[p^n]\to  H^1_{}(k,\Pic_{A/k})[p^n]$$ is surjective. Since $\Pic(A)\to \NS(A_\kbar)$ is surjective, we know that $\pi_0(\Pic_{A/k})$ is a constant finitely generated torsion-free group over $k$ such that $\Pic_{A/k}(k)\to \pi_0(\Pic_{A/k})(k)$ is surjective. Looking at the cohomology long exact sequence associated to $$0\to \Pic_{A/k}^\circ\to \Pic_{A/k}\to \pi_0(\Pic_{A/k})\to 0,$$ we then deduce that $H^1_{}(k,\Pic_{A/k}^\circ)\iso H^1_{}(k,\Pic_{A/k})$, which yields the desired result.
\spa

We now prove that $\hat{C}(A)$ is acyclic. The kernel of $H^2(A,\Zp(1))\to H^2(A_\kbar,\Zp(1))$ is $\varprojlim_n K_{1,n}$ and the kernel of $\Tp(\Br(A))\to \Tp(\Br(A_\kbar))$ is $\Tp(K_2)$. Thus, again, we have to prove that $\varprojlim_n K_{1,n}\to \Tp(K_2)$ is surjective. Combining the previous discussion and the fact that $\Br(k)$ is $p$-divisible, we deduce that the two groups sit in the following diagram with exact rows
	
	\begin{equation*}\begin{tikzcd}
			0\arrow[r] &\Tp(\Br(k))\arrow[d,""]\arrow[r] &  \varprojlim_n K_{1,n}\arrow[d,""]\arrow[r] & \varprojlim_nH^1_{}(k,\Pic_{A/k}[p^n])\arrow[r]\arrow[d,""] & 0\\
			0\arrow[r]& \Tp(\Br(k)) \arrow[r]& \Tp(K_2)\arrow[r]& \Tp(H^1_{}(k,\Pic_{A/k}))\arrow[r]& 0.
		\end{tikzcd}
	\end{equation*}


For every $n>0$, the kernel of $H^1_{}(k,\Pic_{A/k}[p^n]) \to H^1_{}(k,\Pic_{A/k})[p^n]$ is $\Pic(A)/p^n$ and the groups $(\Pic(A)/p^n)_{n>0}$ form a Mittag--Leffler system. We deduce that the morphism $$\varprojlim_{n} H^1_{}(k,\Pic_{A/k}[p^n])\to \Tp(H^1_{}(k,\Pic_{A/k}))$$ is surjective. This implies that $\hat{C}(A)$ is acyclic, as we wanted.
\end{proof}

The proof of the following proposition was inspired by \cite[Thm. 2.1]{CS13}.

\begin{prop}\label{tricky:p}If $\mathrm{char}(k)=p$ and $A$ is an abelian variety over $k$, we have $H^2(A_\kbar,\Zp(1))^{k}[\tfrac{1}{p}]=(\varprojlim_{n}H^2(A_\kbar,\mupn)^{k})[\tfrac{1}{p}]$ and $\Tp(\Br(A_{\kbar}))^{k}[\tfrac{1}{p}]=\Vp(\Br(A_{\kbar})^{k}).$
\end{prop}

\begin{proof} We first note that the four $\Qp$-vector spaces are invariant under isogenies of $A$ and finite separable extension of $k$. Indeed, for every isogeny $\varphi:B\to A$ there exists an isogeny $\psi: A\to B$ such that the composition $\varphi \circ \psi$ is the multiplication by some positive integer $n$. Since $n$ is invertible in $\Qp$, we deduce that $\varphi^*$ is an isomorphism at the level of cohomology groups. Similarly, if $k'/k$ is a finite separable extension, then the pullback morphisms with respect to $A_{k'}\to A$ admit as inverse the morphisms $\tfrac{1}{[k':k]}\mathrm{Tr}_{A_{k'}/A}$.
	
	\spa
Next, thanks to \cite[Thm. 11]{Kat99}, we note that there exists a proper smooth connected curve $C$ with a rational point and a morphism $C\to A$ such that $B:=\mathrm{Jac}(C)$ maps surjectively to $A$. By Poincaré's complete reducibility theorem, $B$ is isogenous to a product $A\times_k A'$ with $A'$ an abelian variety over $k$. Since $H^2(A_\kbar,\Zp(1))^{k}$ (resp. $\Br(A_{\kbar})^k$) is a direct summand of $H^2(A_\kbar\times_\kbar A'_\kbar,\Zp(1))^{k}$ (resp. $\Br(A_\kbar\times_\kbar A'_{\kbar})^k$) and the property we want to prove is invariant by isogenies, it is then enough to prove the result for $B$. In addition, since in the statement it is harmless to extend $k$ to a finite separable extension, we may assume that $\Pic(B)\to \NS(B_{\kbar})$ is surjective, so that $H^1(k,\Pic^\circ_{B/k})=H^1(k,\Pic_{B/k}).$
\spa

Let $K_{1,n}$ be the kernel of the morphism $H^2(B,\mupn)\to H^2(B_\kbar,\mupn)$. By Lemma \ref{fppf-exte-scal:l} and Lemma \ref{limi:l}, the group $K_{1,n}$ is an extension of $H^1(k,\Pic_{B/k}[p^n])$ by $\Br(k)[p^n]$ and by the assumption $\Pic_{C/k}[p^n]=\Pic_{B/k}[p^n]$. We deduce that $K_{1,n}=\ker(H^2(C,\mupn)\to H^2(C_\kbar,\mupn))$. By \cite[Rmq. 2.5.b]{Gro68}, the group $\Br(C_\kbar)$ vanishes, thus $H^2(C_{\kbar},\Zp(1))=\Zp$ and the morphism $H^2(C,\Zp(1))\to H^2(C_{\kbar},\Zp(1))$ is surjective because $C$ has a rational point. This implies that $R^1 \varprojlim_{n}K_{1,n}\to R^1 \varprojlim_{n}H^2(C,\mupn)$ is injective. Since $$R^1 \varprojlim_{n}K_{1,n}\to R^1 \varprojlim_{n}H^2(C,\mupn)$$ factors through $R^1 \varprojlim_{n}H^2(B,\mupn),$ we deduce that $R^1 \varprojlim_{n}K_{1,n}\to R^1 \varprojlim_{n}H^2(B,\mupn)$ is injective as well. Therefore, the morphism $H^2(B,\Zp(1))\to \varprojlim_{n}H^2(B_\kbar,\mupn)^{k}$ is surjective. Thanks to Proposition \ref{Kummer-k-inv:p}, for every $n$ the morphism $H^2(B_\kbar,\mupn)^{k}\to (\Br(B_\kbar)[p^n])^k$ is surjective with finite kernel, so that $\varprojlim_{n}H^2(B_\kbar,\mupn)^{k}\to  \varprojlim_n(\Br(B_\kbar)[p^n])^k$ is surjective as well. This implies that $\Tp(\Br(B_{\kbar}))^{k}\to\varprojlim_n(\Br(B_\kbar)[p^n])^k$ is surjective.

\spa

It remains to prove that for every $n$ we have $(\Br(B_\kbar)[p^n])^k=\Br(B_\kbar)^k[p^n]$. Consider the natural morphism $K_3\to \Br(C)$ where $K_3$ is the kernel of $\Br(B)\to \Br(B_{\kbar})$. Thanks to Lemma \ref{fppf-exte-scal:l} and Lemma \ref{limi:l} and using the fact that $\Br(C_\kbar)=0$, this morphism sits in the following commutative diagram with exact rows
\begin{equation*}\begin{tikzcd}
		0\arrow[r] &\Br(k)\arrow[d,""]\arrow[r] &  K_3\arrow[d,""]\arrow[r] & H^1(k,\Pic_{B/k})\arrow[r]\arrow[d,""] & 0\\
		0\arrow[r]& \Br(k) \arrow[r]& \Br(C)\arrow[r]&H^1(k,\Pic_{C/k}) \arrow[r]& 0.
	\end{tikzcd}
\end{equation*} 

Since $\Pic(B)\to \NS(B_\kbar)$ is surjective and $C$ is a curve, we have that $$H^1(k,\Pic_{B/k})=H^1(k,\Pic^\circ_{B/k})\simeq H^1(k,\Pic^\circ_{C/k})=H^1(k,\Pic_{C/k}).$$

 We deduce that $K_3\to \Br(C)$ is an isomorphism, thus $\Br(B)\to \Br(C)\iso K_3$ provides a splitting of the exact sequence $$0\to K_3\to \Br(B)\to \Br(B_\kbar)^k\to 0.$$ This implies that $\Br(B)[p^n]\to \Br(B_{\kbar})^k[p^n]$ is surjective and this yields the desired result. 
\end{proof}

\section{Constructing a morphism}\label{cons-morp:s}
\label{cons-a-morp:s}
Let $A$ be an abelian variety over a field $k$. For a line bundle $\calL$ of $A$ we write $\varphi_{\calL}:A\to A^\vee$ for the morphism which sends $x\mapsto t_x^*\calL\otimes\calL^{-1}$, where $t_x$ is the translation by $x$. In this section we want to complete the following solid square
	\begin{equation*}
	\begin{tikzcd}\Pic(A)^\wedge\arrow[r,"c_1"]\arrow[d,"\calL\mapsto \varphi_\calL"] & H^2(A,\Zp(1))\arrow[d,dotted]\\
		\Hom(A,A^\vee)_{\Zp}\arrow[r] & \Hom(A[p^\infty],A^\vee[p^\infty]).
	\end{tikzcd}
\end{equation*}	
If $k$ is an algebraically closed field of characteristic $0$ such a commutative diagram is constructed in \cite[Lem. 2.6]{OSZ21} using an analytic method. We propose instead an algebraic construction which works for any field.

\subsection{}
	Consider the morphism $h_1:H^2(A,\mupn)\to H^2(A\times_k A,\mupn)$ which sends a class $\alpha$ to $m^*(\alpha)-\pi_1^*(\alpha)-\pi_2^*(\alpha)$, where $\pi_1$ and $\pi_2$ are the two projections of $A\times_k A$. This morphism has the property that the first Chern class $c_1(\calL)\in H^2(A,\mupn)$ of a line bundle $\calL$ is sent to $c_1(\Lambda(\calL))$, the first Chern class of the associated Mumford bundle $\Lambda(\calL):=m^*\calL\otimes \pi_1^*\calL^{-1}\otimes \pi_2^*\calL^{-1}$. The Leray spectral sequence
\begin{equation}\label{lerayAxA:e}
	E^{i,j}_2:=H^i(A,R^j\pi_{2*}\mupn)\Rightarrow H^{i+j}(A\times_k A,\mupn)
\end{equation}
induces a filtration $0\subseteq F^2 H^{2}(A\times_k A,\mupn) \subseteq F^1 H^{2}(A\times_k A,\mupn)\subseteq H^{2}(A\times_k A,\mupn) $.
\begin{lemm}\label{image:l}The image of $h_1$ lies in $F^1H^{2}(A\times_k A,\mupn)$.
\end{lemm}
\begin{proof}
The spectral sequence (\ref{lerayAxA:e}) gives the exact sequence $$0\to F^1H^{2}(A\times_k A,\mupn)\to H^{2}(A\times_k A,\mupn)\to E^{0,2}_\infty\to 0.$$ Therefore, it is enough to check that the composition $$H^2(A,\mupn)\xrightarrow{h_1} H^2(A\times_k A,\mupn)\to H^0(A,R^2\pi_{2*}\mupn)$$ is the $0$-morphism. By \cite[Cor. 1.4]{BO21}, there exists a commutative linear algebraic group\footnote{Recall that a \textit{linear algebraic group} over $k$ is an affine group scheme of finite type over $k$.} $G$ over $k$ which represents $R^2q_{*}\mupn$ on the big fppf site $(\mathbf{Sch}/k)_{\fppf}$. Since $R^2\pi_{2*}\mupn$ is the restriction of $R^2q_{*}\mupn$ from $(\mathbf{Sch}/k)_{\fppf}$ to $(\mathbf{Sch}/A)_{\fppf}$, this implies that $H^0(A,R^2\pi_{2*}\mupn)$ can be computed as $\Mor_{\mathbf{Sch}/k}(A,G).$ Thanks to the fact that $G$ is affine, every morphism $A\to G$ contracts $A$ to a point. We deduce that $\Mor_{\mathbf{Sch}/k}(A,G)=\Mor_{\mathbf{Sch}/k}(0_A,G)=H^0(k,R^2q_{*}\mupn)$. By Lemma \ref{limi:l}, the group $H^0(k,R^2q_{*}\mupn)$ is naturally a subgroup of $H^2(A_\kbar,\mupn)$ and the induced morphism $$H^2(A\times_k A,\mupn)\to H^0(k,R^2q_{*}\mupn)\hookrightarrow H^2(A_\kbar,\mupn)$$ is given by the pullback via $i_1:A=A\times_k 0_A \hookrightarrow A\times_k A$ followed by the extension of scalars to $\kbar$. By construction, we have that $i_1^*\circ h_1=i_1^*\circ m^*-i_1^*\circ \pi_1^*=0$. This concludes the proof.
\end{proof}

\begin{lemm}\label{HomH1:l}
	Let $G$ be a finite commutative group scheme killed by a positive integer $n$. There is a natural injective morphism $f_n:\Hom(A[n],G)\to H^1(A,G)$ which admits a retraction $g_n$.
\end{lemm}

\begin{proof}
	Write $P_{n}$ for the $A[n]$-torsor over $A$ given by the multiplication by $n$. The morphism $f_n$ is then defined by $f_n(\sigma):=\sigma_*P_{n}$ for every $\sigma\in \Hom(A[n],G)$. We want to define now $g_n$ which sends a $G$-torsor $P$ over $A$ to an homomorphism $g_n(P): A[n]\to G$. By Cartier duality, this is the same as defining a morphism $g_n(P)^\vee:G^\vee\to (A[n])^\vee=A^\vee[n]$. For a scheme $T$ over $k$ and a $T$-point of $G^\vee$ corresponding to a morphism $\tau:G_T\to \mathbb{G}_{m,T}$ we define  $g_n(P)^\vee(\tau)$ as $\tau_{*}P_T\in H^1(A_T,\mathbb{G}_{m,T})[n]=A^\vee[n](T)$. To prove that $g_n\circ f_n=\id$ it is enough to note that for every $\sigma \in \Hom(A[n],G)$, every scheme $T$ over $k$, and every $\tau\in \Hom(G_T,\mathbb{G}_{m,T})$ we have that  $g_n(f_n(\sigma))^\vee(\tau)=\tau_*(\sigma_*P_n)_T=(\tau\circ \sigma_T)_*P_{n,T}$ is the line bundle over $A_T$ associated to $\tau\circ \sigma_T\in (A[n])^\vee(T)$ under the identification $(A[n])^\vee=A^\vee[n]$.
\end{proof}

\subsection{}

Thanks to Lemma \ref{image:l}, we can define $$\bar{h}_1: H^2(A,\mupn)\to H^1(A,A^\vee[p^n])$$

as the composition of $h_1$ and the natural morphism $$F^1H^2(A\times_k A,\mupn)\to H^1(A,R^1\pi_{2*}\mupn)=H^1(A,A^\vee[p^n]).$$ In addition, by Lemma \ref{HomH1:l} applied to $G=A^\vee[p^n]$, we get a morphism $$h_2:H^1(A, A^\vee[p^n])\to\Hom(A[p^n],A^\vee[p^n]).$$ We write $$h: H^2(A,\mupn)\to \Hom(A[p^n],A^\vee[p^n])$$ for the composition ${h_2}\circ\bar{h}_1$ and we denote with the same letter the induced morphism $H^2(A,\Zp(1))\to \Hom(A[p^\infty],A^\vee[p^\infty])$.

\begin{prop}\label{comm-squa:p}
The square
		\begin{equation}
	\begin{tikzcd}
		\Pic(A)/p^n \arrow[d,""]\arrow[r,"c_1"] & H^2_{}(A,\mupn) \arrow[d,"h"]\\
		\Hom(A,A^\vee)/p^n \arrow[r]& \Hom(A[p^n],A^\vee[p^n])
	\end{tikzcd}
\end{equation}

is commutative.
\end{prop}
\begin{proof}
We have to show that for every line bundle $\calL\in \Pic(A)$ we have $$h(c_1(\calL))=\varphi_{\calL}|_{A[p^n]}.$$ Consider the Leray spectral sequence 
\begin{equation}\label{lerayAxAv:e}
E^{i,j}_2=H^i(A^\vee,R^j\pi_{2*}\mupn)\Rightarrow H^{i+j}(A\times_k A^\vee,\mupn).
\end{equation}

The morphism $A\times_k A\xrightarrow{\id_A\times \varphi_{\calL}} A\times_k A^\vee$ induces via pullback a morphism from (\ref{lerayAxAv:e}) to (\ref{lerayAxA:e}). This produces the following commutative diagram

\begin{equation*}\begin{tikzcd}
		&F^1H^2(A\times_k A^\vee, \mupn)\arrow[r]\arrow[d, "(\id_A\times \varphi_\calL)^*"] &H^1(A^\vee,A^\vee[p^n]) \arrow[d, "\varphi_\calL^*"]&  \\
	H^2(A, \mupn)\arrow[r,"h_1"] &
		F^1H^2(A\times_k A, \mupn)\arrow[r] &H^1(A,A^\vee[p^n])\arrow[r,"h_2"] &  \Hom(A[p^n],A^\vee[p^n])
	\end{tikzcd}
\end{equation*} 

where the composition of the lower horizontal arrows is $h$. If $\calP\in \Pic({A\times_k A^\vee})$ is the Poincaré bundle of $A$, we have that $\Lambda(\calL)=(\id_A\times \varphi_{\calL})^* \calP$. This implies that $h_1(c_1(\calL))=c_1(\Lambda(\calL))=(\id_A\times \varphi_{\calL})^* c_1(\calP)$. In addition, by direct inspection, we note that $h_2(\varphi_{\calL}^*([P]))=\varphi_\calL|_{A[p^n]}$, where $[P]\in H^1(A^\vee,A^\vee[p^n])$ is the class of the torsor $A^\vee \xrightarrow{\cdot p^n} A^\vee$. It remains to prove that the morphism $F^1H^2(A\times_k A^\vee,\mupn)\to H^1(A^\vee,A^\vee[p^n])$ sends $c_1(\calP)$ to $[P]$. For this purpose, we introduce the Leray spectral sequence
 \begin{equation}\label{lerayGm:e}
 	E^{i,j}_2=H^i(A^\vee,R^j\pi_{2*}\Gm[1])\Rightarrow H^{i+j}(A\times_k A^\vee,\Gm[1]).
 \end{equation}
 
 The morphism $\delta:\Gm[1]\to \mupn$ associated to the Kummer exact sequence induces a morphism from (\ref{lerayGm:e}) to (\ref{lerayAxAv:e}) which we denote with the same symbol. In turn, this produces the commutative diagram 
\begin{equation*}\begin{tikzcd}
		H^1(A\times_k A^\vee, \Gm)=F^1H^1(A\times_k A^\vee, \Gm)\arrow[r]\arrow[d, "c_1"] &H^0(A^\vee,A^\vee) \arrow[d, "\delta"].&\\
		F^1H^2(A\times_k A^\vee, \mupn)\arrow[r] &H^1(A^\vee,A^\vee[p^n])		
	\end{tikzcd}
\end{equation*} 

The upper horizontal arrow sends the line bundle $\calP$ to $\id_{A^\vee}\in H^0(A^\vee,A^\vee)$, while $\delta$ sends $\id_{A^\vee}$ to $[P]$. This yields the desired result.
\end{proof}

\begin{prop}\label{imag:p}
The morphism $h:H^2(A_\kbar,\Zp(1))\to \Hom_{}(A_\kbar[p^\infty],A^\vee_\kbar[p^\infty])$ is an injective morphism with image $\Hom^\sym_{}(A_\kbar[p^\infty],A^\vee_\kbar[p^\infty])$, the group of homomorphisms which are fixed by the involution $\tau\mapsto \tau^\vee$.
\end{prop}

\begin{proof}Suppose $\mathrm{char}(k)=p$ and write $W$ for the ring of Witt vectors of $\kbar$. The crystalline cohomology groups of an abelian variety are torsion free by \cite[Cor. 2.5.5]{BBM82}. Therefore, thanks to the K\"unneth formula, \cite[Thm. V.4.2.1]{Ber74}, we have that $H^*_{\crys}(A_\kbar\times_\kbar A_\kbar/W)=H^*_{\crys}(A_\kbar/W)\otimes H^*_{\crys}(A_\kbar/W)$ so that $m:A\times_k A\to A$ induces a morphism $$m^*:H^*_{\crys}(A_\kbar/W)\to H^*_{\crys}(A_\kbar/W)\otimes H^*_{\crys}(A_\kbar/W).$$ 
In degree $2$ we get a morphism $$m^*:H^2_{\crys}(A_\kbar/W)\to H^2_{\crys}(A_\kbar/W)\oplus H^1_{\crys}(A_\kbar/W)^{\otimes 2} \oplus H^2_{\crys}(A_\kbar/W)$$ which in turn induces a morphism $$m^*-\pi_1^*-\pi_2^*: H^2_{\crys}(A_\kbar/W)\to H^1_{\crys}(A_\kbar/W)^{\otimes 2}.$$

Write $\sigma: \bigwedge^2 H^1_{\crys}(A_\kbar/W)\iso H^2_{\crys}(A_\kbar/W)$ for the natural isomorphism induced by the cup product, as in \cite[Cor. 2.5.5]{BBM82}. For every $v\in H^1_{\crys}(A_\kbar/W)$, the pullback $m^*(v)$ is equal to $\pi_1^*(v)+\pi_2^*(v)$. Therefore, the composition $(m^*-\pi_1^*-\pi_2^*)\circ \sigma: \bigwedge^2 H^1_{\crys}(A_\kbar/W)\hookrightarrow H^1_{\crys}(A_\kbar/W)^{\otimes 2}$ is equal to the natural embedding $v\wedge w\mapsto v\otimes w- w\otimes v$. By [\textit{ibid.}, Thm. 5.1.8], we have that the $F$-crystal $H^1_{\crys}(A_\kbar/W)$ over $\kbar$ is canonically isomorphic to $H^1_{\crys}(A_\kbar^\vee/W)^\vee$ with $F$-structure defined as the dual of the $F$-structure of $H^1_{\crys}(A_\kbar^\vee/W)$ multiplied by $p$. Thus, we have that $$(H^1_{\crys}(A_\kbar/W)^{\otimes 2})^{F=p}=\Hom_{\mathbf{F\textbf{-}Crys}(\kbar)}(H^1_{\crys}(A_\kbar^\vee/W),H^1_{\crys}(A_\kbar/W)),$$ where $\mathbf{F\textbf{-}Crys}(\kbar)$ is the category of $F$-crystals over $\kbar$. By \cite[Rmq. II.3.11.2]{Ill79}, the $F$-crystals $H^1_{\crys}(A_\kbar/W)$ and $H^1_{\crys}(A_\kbar^\vee/W)$ are the contravariant crystalline Dieudonné modules of the $p$-divisible groups $A_\kbar[p^\infty]$ and $A^\vee_\kbar[p^\infty]$, thus we get $$(H^1_{\crys}(A_\kbar/W)^{\otimes 2})^{F=p}=\Hom(A_\kbar[p^\infty],A^\vee_\kbar[p^\infty]).$$ On the other hand, by [\textit{ibid.}, Thm. II.5.14], there is a canonical isomorphism $H^2(A_\kbar,\Zp(1))=H^2_{\crys}(A_\kbar/W)^{F=p}$. This concludes the case when $\mathrm{char}(k)=p$. If $p$ is invertible in $k$ one can replace crystalline cohomology with $p$-adic étale cohomology.
\end{proof}
\section{Main results}\label{main-resu:s}

We are now ready to prove our main result, which is a flat version of the Tate conjecture for divisors of abelian varieties.

\begin{theo}\label{fppf-tate:t}
	If $A$ is an abelian variety over a finitely generated field $k$ of characteristic $p>0$, then $\Tp(\Br(A_\kbar)^k)=0$. Moreover, after possibly replacing $k$ with a finite separable extension, the cycle class map $$c_1:\NS(A)_{\Zp}\to\varprojlim_n H^2(A_\kbar,\mupn)^k$$ becomes an isomorphism.
\end{theo}
\begin{proof} 
	To prove the statement we may assume that $\Pic(A)\to \NS(A_{\kbar})$ is surjective by extending $k$. 
The $\Zp$-module $\Hom(A[p^\infty],A^\vee[p^\infty])$ embeds into $\Hom(A_\kbar[p^\infty],A^\vee_\kbar[p^\infty])$, therefore the morphism $$h:H^2(A,\Zp(1))\to \Hom(A[p^\infty],A^\vee[p^\infty])$$ induces a morphism $\tilde{h}:H^2(A_\kbar,\Z_p(1))^{k}\to \Hom(A[p^\infty],A^\vee[p^\infty])$. By Proposition \ref{imag:p}, we know that $\tilde{h}$ is injective and $\mathrm{im}(\tilde{h})$ is contained in $\Hom^\sym(A[p^\infty],A^\vee[p^\infty])$.
	In addition, by Proposition \ref{comm-squa:p}, we have the following commutative square	
	\begin{equation*}
		\begin{tikzcd}
			\NS(A)_{\Zp}\arrow[r,"c_1",hook]\arrow[d,hook] & H^2(A_\kbar,\Z_p(1))^{k}\arrow[d,"\tilde{h}",hook]\\
			\Hom(A,A^\vee)_{\Zp}\arrow[r] & \Hom(A[p^\infty],A^\vee[p^\infty]).
		\end{tikzcd}
	\end{equation*}
The lower arrow is an isomorphism by \cite[Thm. 2.6]{deJ98}, and since $\NS(A)=\Hom^\sym(A,A^\vee)$, we deduce that $\mathrm{im}(\tilde{h}\circ c_1)=\Hom^\sym(A[p^\infty],A^\vee[p^\infty])$. This implies that $$c_1:\NS(A)_\Zp\to H^2(A_\kbar,\Z_p(1))^{k}$$ is surjective, thus by Proposition \ref{tricky:p} we get that $$c_1:\NS(A)_\Qp\to \varprojlim_n H^2(A_\kbar,\mupn)^k[\tfrac{1}{p}]$$ is surjective. Combining this with Proposition \ref{Kummer-k-inv:p}, we deduce that $$c_1:\NS(A)_\Zp\to \varprojlim_n H^2(A_\kbar,\mupn)^k$$ is surjective. For the result about the Brauer group, we just note that by the previous argument and Proposition \ref{Kummer-k-inv:p}, the $\Zp$-module $\Tp(\Br(A_\kbar))^k$ vanishes. Therefore, thanks to Proposition \ref{tricky:p}, we deduce that $\Tp(\Br(A_\kbar)^k)=\Tp(\Br(A_\kbar)^k)[\tfrac{1}{p}]=0$.\end{proof}


\begin{theo}\label{fini-Brau:t}
Let $A$ be an abelian variety over a finitely generated field $k$ of characteristic $p>0$. The transcendental Brauer group $\Br(A_{\ks})^{k}$ is a direct sum of a finite group and a finite exponent $p$-group. In addition, if the Witt vector cohomology group $H^2(A_{\kbar},W\calO_{A_\kbar})$ is a finite $W(\kbar)$-module, then $\Br(A_{\ks})^{k}$ is finite.
\end{theo}
\begin{proof}By \cite[Thm. 16.2.3]{CS21}, the group $\Br(A_{\ks})^{k}[\tfrac{1}{p}]$ is finite. Moreover, thanks to Corollary \ref{Br-inj:c}, the morphism $\Br(A_{\ks})\to \Br(A_{\kbar})$ is injective, which implies that the transcendental Brauer group is the same as $\Br(A_{\kbar})^{k}$. Write $\mathbf{Ab}_p^\star\subseteq \mathbf{Ab}$ for the full subcategory of the category of abstract abelian groups with objects those (possibly infinite) $p$-groups isomorphic to $(\Qp/\Zp)^{\oplus a} \oplus M$ for some $a\geq 0$ and $M$ a finite exponent $p$-group. Equivalently, $\mathbf{Ab}_p^\star$ is the subcategory of those $p$-groups $M$ such that $M[p^{n+1}]/M[p^n]$ is finite for $n$ big enough. Note that this subcategory is closed under the operation of taking subobjects, quotients, and finite direct sums. We first want to prove that $H:=\varinjlim_n H^2(A_\kbar,\mupn)\in \mathbf{Ab}_p^\star$ and when $H^2(A_{\kbar},W\calO_{A_\kbar})$ is a finite $W(\kbar)$-module then, in addition, $H[p]$ is finite (so that $H[p^n]$ is also finite for every $n\geq 0$). Thanks to the Kummer exact sequence this implies the same result for $\Br(A_{\kbar})^{k}[p^\infty]$.
	
	\spa
	
	Write $q: A_\kbar \to \Spec(\kbar)$ for the structural morphism. By \cite[Cor. 1.4]{BO21}, for every $n$ there exists a commutative linear algebraic group $G_n$ representing $R^2q_{*}\mu_{p^n}$\footnote{One could use alternatively perfect groups rather than algebraic groups, as in \cite[Lem. 1.8]{Mil86}.}. Write $U_n$ for the unipotent radical of $G_n$ and $D_n$ for the reductive quotient $G_n/U_n$. Since $G_n$ is commutative, there is a canonical Levi decomposition $G_n=U_n\times D_n$. In particular, we have that $H=U\times D$, where $U:=\varinjlim_nU_n(\kbar)$ and $D:=\varinjlim_nD_n(\kbar)$. For every $n>0$, the group scheme $D_n$ is finite, because it is a reductive group killed by $p^n$. In addition, by \cite[Prop. 10.7]{BO21}, there is a canonical isomorphism of formal groups $\varinjlim_n \hat{G}_n=\varinjlim_n \hat{U}_n=\Phi^2_{\mathrm{fl}}(A_\kbar,\Gm)$, where $\hat{G}_n$ and $\hat{U}_n$ are the formal completions at the identity of $G_n$ and $U_n$ and $\Phi^2_{\mathrm{fl}}(-,-)$ is as in [\textit{ibid.},  §10.6].

	\spa

Applying $Rq_*$ to the exact sequence \begin{equation}\label{mupnmup:e}
1\to \mupn \to \mu_{p^{n+1}}\xrightarrow{\cdot p^n} \mu_p\to 1
	\end{equation} and using the fact that $\varinjlim_n R^1q_{*}\mu_{p^n}=\Pic_{A_{\kbar}/\kbar}[p^\infty]$ is a $p$-divisible group, we get the exact sequence $$1\to G_n\to G_{n+1}\xrightarrow{\cdot p^n} G_1.$$

As a first consequence, we deduce that for every $n>0$ the group scheme $D_n$ is the same as $D_{n+1}[p^n]$, thus $D[p]=D_1(\kbar)$ is finite. In particular, the abstract group $D$ is in $\mathbf{Ab}_p^\star$. To bound $U$, we note that by \cite[Prop. 3.1]{Mil86} the dimension of the chain of algebraic groups $G_1\subseteq G_2\subseteq \dots$ is eventually constant. Therefore, there exists $N>0$ such that for every $n\geq N$, the morphism $(U_n)_{\mathrm{red}}\to (U_{n+1})_{\mathrm{red}}$ is an isomorphism. This shows that $U$ is a finite exponent $p$-group. 

\spa

	 If $H^2(A_\kbar,W\calO_{A_\kbar})$ is a finite $W(\kbar)$-module, then the formal group $\Phi^2_{\mathrm{fl}}(A_\kbar,\Gm)$ does not contain any copy of $\hat{\mathbb{G}}_a$. Indeed, by \cite[Cor. 12.5]{BO21}, the group $H^2(A_\kbar,W\calO_{A_\kbar})$ is the Cartier module of $\Phi^2_{\mathrm{fl}}(A_\kbar,\Gm)$ and, by the assumption, it can not contain $\kbar[[V]]$, the Cartier module of $\hat{\mathbb{G}}_a$. Therefore, in this case, we have that each group $U_n(\kbar)$ is trivial, so that $H[p]=D[p]$ is finite.
	
	\spa
	
	 We can finally prove that $\Br(A_{\kbar})^{k}[p^\infty]$ has finite exponent. Suppose by contradiction that this is not the case. Since $\Br(A_{\kbar})^{k}[p^\infty]\in \mathbf{Ab}_p^\star$, we deduce that it contains a copy of $\Qp/\Zp$. On the other hand, by Theorem \ref{fppf-tate:t}, the group $\Tp(\Br(A_{\kbar})^{k})$ vanishes, which leads to a contradiction.
\end{proof}

\begin{coro}\label{galo-fixe:c}
The group $\Br(A_{\ks})^{\Gamma_k}$ has finite exponent.
\end{coro}
\begin{proof}
This follows from Theorem \ref{fini-Brau:t} thanks to \cite[Thm. 5.4.12]{CS21}.
\end{proof}
We end this section with some examples of abelian varieties over finitely generated fields with infinite transcendental Brauer group. Let $E$ be a supersingular elliptic curve over an infinite finitely generated field $k$ and let $A$ be the product $E\times_k E$.

\begin{prop}\label{counterexample:p}
After possibly extending $k$ to a finite separable extension, the transcendental Brauer group $\Br(A_{\ks})^k$ becomes infinite.
\end{prop}

\begin{proof}
Even in this case we use that, thanks to Corollary \ref{Br-inj:c}, the transcendental Brauer group is the same as $\Br(A_{\kbar})^{k}$. Moreover, after extending the scalars we may assume that the morphism $\Pic(A)\to \NS(A_{\kbar})$ is surjective. Combining Proposition \ref{Kummer-k-inv:p} and the fact that $\NS(A_{\kbar})/p$ is finite we deduce that it is enough to show that $H^2(A_{\kbar},\mu_p)^k$ is infinite. We look at the Leray spectral sequence with respect to the second projection $\pi_2:A=E\times_k E\to E$ (both over $k$ and over $\kbar$). In the second page, we have that the boundary morphism $H^1(E,R^1\pi_{2*} \mu_p) \to H^3 (E,\pi_{2*}\mu_p)$ vanishes because $H^3 (E,\pi_{2*}\mu_p)\to H^3(A,\mu_p)$ admits a retraction induced by the zero section of $\pi_2$. Since $H^0(E_\kbar,R^2\pi_{2*}\mu_p)=H^2(E_\kbar,\mu_p)=\Z/p$, it is then enough to show that the image of $$H^1(E,E[p])=H^1(E,R^1\pi_{2*}\mu_p)\to H^1(E_\kbar,R^1\pi_{2*}\mu_p)=H^1(E_\kbar,E_\kbar[p])$$ is infinite. By Lemma \ref{HomH1:l}, we have that $\End(E[p])$ (resp. $\End(E_{\kbar}[p])$) admits a natural embedding in $H^1(E,E[p])$ (resp. $H^1(E_{\kbar},E_{\kbar}[p])$). Since $$k=\End(\alpha_p)\subseteq \End(E[p])\subseteq \End(E_{\kbar}[p])$$ by the assumption that $E$ is supersingular, we deduce the desired result.
\end{proof}

\section{Specialisation of Néron--Severi groups}
\label{spec:s}

\subsection{}\label{failure:ss}
We want to start this section with an explicative example. Let $\calE\to X$ be a non-isotrivial family of ordinary elliptic curves, where $X$ is a connected normal scheme of finite type over $\Fp$. Let $\calA$ be the fibred product $\calE\times_X \calE$. We denote by $E$ and $A$ the generic fibres over the generic point $\Spec(k)\hookrightarrow X$. The Kummer exact sequence induces the exact sequence 
	$$0\to\mathrm{NS}(A_{\bar{k}})_{\Z_p} \to H^2_{}(A_{\bar{k}},\Z_p(1))\to \Tp(\mathrm{Br}(A_{\bar{k}}))\to 0.$$

The group $\mathrm{NS}(A_{\bar{k}})_{\Z_p}$ is of rank $2+\rk_\Z(\End(E_{\bar{k}}))=3$, while, by Proposition \ref{imag:p}, the rank of $H^2_{}(A_{\bar{k}},\Z_p(1))$ is $2+\rk_\Zp(\End(E_\kbar[p^\infty]))=4$. This shows that $\Tp(\mathrm{Br}(A_{\bar{k}}))$ is of rank $1$. The endomorphisms of $E_{\kbar}[p^\infty]$ are all defined over $\ki$, which implies that the action of $\Gamma_k$ on $H^2_{}(A_{\bar{k}},\Z_p(1))$ is trivial. In particular, the morphism $\mathrm{NS}(A)_{\Z_p} \to H^2_{}(A_{\bar{k}},\Z_p(1))^{\Gamma_k}$ is not surjective and the cokernel $\Tp(\mathrm{Br}(A_{\bar{k}}))^{\Gamma_k}$ is isomorphic to $\Zp$.

\spa

In this case, the Galois action on $H^2_{}(A_{\bar{k}},\Z_p(1))$ is not enough to detect what classes are $\Zp$-linear combinations of algebraic cycles. There is an additional obstruction to descend cohomology classes through the purely inseparable extension $\ki/k$. This extra purely inseparable obstruction gives an explanation of the failure of surjectivity of specialisation morphisms of Néron--Severi groups. In the example, if $\Spec(\kappa)\hookrightarrow X$ is a closed point, we have that $\NS(\calA_{{\kappa}})=\NS(\calA_{\bar{\kappa}})$ is of rank $4$ because $\End(\calE_\kappa)=\End(\calE_{\bar\kappa})$ is of rank $2$ (there is an extra Frobenius endomorphism). Thus, the specialisation map $\NS(A)\hookrightarrow \NS(\calA_\kappa)$ is never surjective even if the rank of $H^2_{}(A_{\bar{\kappa}},\Z_p(1))$ is $4$ as the generic geometric fibre and the Galois action is trivial in both cases. One can interpret this failure by saying that the extra obstruction on $H^2_{}(A_{\bar{k}},\Z_p(1))$ coming from the purely inseparable extension $\ki/k$ is trivial on $H^2_{}(A_{\bar{\kappa}},\Z_p(1))$ since $\kappa$ is perfect. In general, we prove the following theorem.

\begin{theo}\label{NS:t}
	Let $X$ be a connected normal scheme of finite type over $\Fp$ with generic point $\eta=\Spec(k)$ and let $f:\calA\to X$ be an abelian scheme over $X$ with constant Newton polygon. For every closed point $x=\Spec(\kappa)$ of $X$ we have $$\rk_\Z(\NS(\calA_{\bar{x}})^{\Gamma_\kappa})-\rk_\Z(\NS(\calA_{\bar{\eta}})^{\Gamma_k})\geq \rk_\Zp ( \Tp(\Br(\calA_{\bar{\eta}}))^{\Gamma_k}).$$
\end{theo}

\begin{rema}
Note that after replacing $X$ with a finite étale cover the action of $\Gamma_k$ on $\NS(\calA_{\bar{\eta}})$ is trivial. Thus we also get an inequality before taking Galois-fixed points.
\end{rema}
To prove Theorem \ref{NS:t} we first need the following result.

\begin{prop}\label{lisse-sheaf:p}
	Under the assumptions of Theorem \ref{NS:t}, the functor $\calF$ which sends $T\in (X^{\mathrm{perf}})_{\mathrm{pro}\et}$\footnote{We denote by $(-)_{\mathrm{pro}\et}$ the pro-étale site of a scheme, as defined in \cite{BS15}.} to $\Hom^{\sym}(\calA_T[p^\infty],\calA^\vee_T[p^\infty])[\tfrac{1}{p}]$ is a semi-simple finite rank $\Qp$-local system such that for every $\bar{x}\in X(\bar{\mathbb{F}}_p)$ we have $$\calF_{\bar{x}}=\Hom^{\sym}(\calA_{\bar{x}}[p^\infty],\calA^\vee_{\bar{x}}[p^\infty])[\tfrac{1}{p}].$$
\end{prop}
\begin{proof} Let $\widetilde{\calF}$ be the functor which sends $T\in (X^{\mathrm{perf}})_{\mathrm{pro}\et}$ to $\Hom(\calA_T[p^\infty],\calA^\vee_T[p^\infty])[\tfrac{1}{p}]$. We first note that to prove the result we can replace $\calF$ with $\widetilde{\calF}$ since $\calF$ is the kernel of the $\Qp$-linear endomorphism $\alpha-\id_{\widetilde{\calF}}:\widetilde{\calF}\to \widetilde{\calF}$ where $\alpha$ sends $\tau\in \Hom(\calA_T[p^\infty],\calA^\vee_T[p^\infty])[\tfrac{1}{p}]$ to $\tau^\vee$.
	Write $\Fcrys(X)$ for the category of $F$-crystals over the absolute crystalline site of $X$ and let $\calM_1,\calM_2\in \Fcrys(X)$ be the contravariant crystalline Dieudonné modules of $\calA[p^\infty]$ and $\calA^\vee[p^\infty]$ over $X$ constructed in \cite[Déf. 3.3.6]{BBM82}. By [\textit{ibid.}, Thm. 5.1.8], we have that $\calM_1= \calM_2^\vee (-1)$ where $\calM_2^\vee(-1)$ is the $F$-crystal $\HHom(\calM_2,\calO_{X,\mathrm{cris}})$ endowed with the dual of the $F$-structure of $\calM_2$ multiplied by $p$. Thus $\calM_1^{\otimes 2}$ is equal to $\HHom(\calM_2,\calM_1)$ endowed with the natural $F$-structure multiplied by $p$. By \cite[Thm. D]{Lau13}, for every perfect scheme $T\to X$ we have canonical isomorphisms $$\Gamma(T,\calM_1^{\otimes 2})^{F=p}=\Hom_{\Fcrys(T)}(\calM_{2,T},\calM_{1,T})=\Hom(\calA_T[p^\infty],\calA^\vee_T[p^\infty])$$ where $\calM_{1,T}$ and $\calM_{2,T}$ are the inverse images of $\calM_1$ and $\calM_2$ to $T$. These isomorphisms are equivariant with respect to the action of the abstract group $\Aut(T/X)$.

	 \spa
	 
	 By \cite[Thm. 2.5.1]{Katz79}, the slope filtration of the $F$-crystal  $\calM_{1,X^{\textrm{perf}}}^{\otimes 2}$ (which exists since $\calA\to X$ has constant Newton polygon) splits uniquely up to isogeny. We denote by $\calN^{[1]}_{X^\textrm{perf}}$ the slope $1$ subobject of $\calM_{1,X^{\textrm{perf}}}^{\otimes 2}$, defined up to isogeny. Note that for every $T\to X^\textrm{perf}$ we have that $$\Gamma(T,\calM_{1,T}^{\otimes 2})^{F=p}=\Gamma(T,\calN^{[1]}_T)^{F=p}=\Gamma(T,\calN^{[1]}_{T}(1))^{F=1}.$$
	 
	 By construction, the $F$-crystal $\calN^{[1]}_{X^{\textrm{perf}}}(1)$ is unit-root. Therefore, by \cite[Prop. 4.1.1]{Kat73}, we deduce that $\widetilde{\calF}$ is a $\Qp$-local system. In addition, by [\textit{ibid.}, Lem. 4.3.15], for every $S=\Spec(R)\to X$ with $R$ strictly henselian perfect ring we have $$\Hom(\calA_{S}[p^\infty],\calA^\vee_{S}[p^\infty])[\tfrac{1}{p}]=\Hom(\calA_{s}[p^\infty],\calA^\vee_{s}[p^\infty])[\tfrac{1}{p}]$$ where $s$ is the closed point of $S$. This implies that for every $\bar{x}\in X(\bar{\mathbb{F}}_p)$ we have $$\calF_{\bar{x}}=\Hom^{\sym}(\calA_{\bar{x}}[p^\infty],\calA^\vee_{\bar{x}}[p^\infty])[\tfrac{1}{p}].$$



	\spa
	
	For the semi-simplicity, since $X$ is normal, we can shrink $X$ and assume it smooth. Write $\calN$ for the $F$-isocrystal $(R^1f_{\crys*}\calO_{\calA,\crys})^{\otimes 2}$ and $\calN^{[1]}$ for the quotient $\calN^{\leq 1}/\calN^{<1}$, where $\calN^{\leq 1}$ (resp. $\calN^{<1}$) is the subobject of $\calN$ of slopes $\leq 1$ (resp. $<1$). Note that by \cite[Thm. 2.5.6.(ii)]{BBM82}, the pullback of $\calN^{[1]}$ to $X^{\mathrm{perf}}$ is isomorphic as an $F$-isocrystal with $\calN^{[1]}_{X^\textrm{perf}}$ over $X^{\textrm{perf}}$ defined above. Thanks to \cite[Thm. 1.1.2]{D'Ad20}, we have that $\calN^{[1]}$ is semi-simple as an $F$-isocrystal.

	\spa

	By \cite[Thm. 2.1]{Cre87}, there is an equivalence between unit-root $F$-isocrystals over $X$ and finite rank $\Qp$-local systems. By construction, Crew's and Katz's correspondences are compatible, in the sense that they agree after pulling back the objects through $X^\mathrm{perf}\to X$. Since the étale fundamental groups of $X$ and $X^{\mathrm{perf}}$ are canonically isomorphic, we deduce that $\widetilde{\calF}$ is semi-simple as well. This yields the desired result.

\end{proof}

\subsection{}\textit{Proof of Theorem \ref{NS:t}. }We look at the exact sequence
$$0\to\NS(A_{\kbar})_{\Qp}\to H^2(A_{\kbar},\Qp(1))\to \Tp(\Br(A_{\kbar}))[\tfrac{1}{p}]\to 0.$$

Thanks to Proposition \ref{lisse-sheaf:p}, the operation of taking Galois-fixed points is exact. We get the exact sequence
$$0\to\NS(A_{\kbar})_{\Qp}^{\Gamma_k}\to H^2(A_{\kbar},\Qp(1))^{\Gamma_k}\to \Tp(\Br(A_{\kbar}))[\tfrac{1}{p}]^{\Gamma_k}\to 0.$$
Looking at the ranks we deduce the following equality
\begin{equation}\label{eq:e}
\rk_\Zp(H^2(A_{\kbar},\Zp(1))^{\Gamma_k})=\rk_\Z(\NS(A_{\kbar})^{\Gamma_k})+\rk_\Zp(\Tp(\Br(A_{\kbar}))^{\Gamma_k}).
\end{equation}
By Proposition \ref{lisse-sheaf:p}, the action of $\Gamma_k$ on $H^2(A_{\kbar},\Qp(1))=\Hom^\sym(A_\kbar[p^\infty],A^\vee_\kbar[p^\infty])[\tfrac{1}{p}]$ factors through the étale fundamental group of $X$ associated to $\bar{\eta}$, denoted by $\pi_1^{\et}(X,\bar{\eta})$. In addition, if $\kappa$ is the residue field of $x$, the inclusion $x\hookrightarrow X$ induces then an action of $\Gamma_\kappa$ on $H^2(A_{\kbar},\Qp(1))$ which corresponds, up to conjugation, to the action of $\Gamma_\kappa$ on $H^2(A_{\bar{\kappa}},\Qp(1))$. Therefore, by the Tate conjecture over finite fields (or Corollary \ref{galo-fixe:c}), we get $\NS(\calA_{\bar{x}})^{\Gamma_\kappa}_\Qp=H^2(A_{\kbar},\Qp(1))^{\Gamma_\kappa}$. Since $H^2(A_{\kbar},\Qp(1))^{\Gamma_k}$ is a subspace of $H^2(A_{\kbar},\Qp(1))^{\Gamma_\kappa}$ we deduce that
\begin{equation}\label{ineq:e}
	\rk_\Z(\NS(\calA_{\bar{x}})^{\Gamma_\kappa})=\rk_\Zp(H^2(A_{\kbar},\Zp(1))^{\Gamma_\kappa})\geq \rk_\Zp(H^2(A_{\kbar},\Zp(1))^{\Gamma_k}).
\end{equation}

Combining (\ref{eq:e}) and (\ref{ineq:e}) we get the desired result.\qed

\spa

We want to conclude this section with other examples of abelian varieties such that $ \Tp(\Br(A_{\kbar}))^{\Gamma_k}\neq 0$. These are variants of the abelian surface of §\ref{failure:ss} and they all provide counterexamples to the conjecture in \cite[§7.3.1]{Ulm14} when $\ell=p$. 
\begin{prop}\label{counter:p}
	Let $A$ be an abelian variety which splits as a product $B\times_k B$ with $B$ an abelian variety over $k$. There is a natural exact sequence $$0\to \Hom(B,B^\vee)_\Zp\to \Hom(B_\kbar[p^\infty],B_\kbar^\vee[p^\infty])^{\Gamma_k}\to \Tp(\Br(A_{\kbar}))^{\Gamma_k}.$$

\end{prop}

\begin{proof}
	We consider the exact sequence $$0\to \NS(A_{\kbar})^{\Gamma_k}_\Zp\to H^2(A_{\kbar}, \Zp(1))^{\Gamma_k}\to \Tp(\Br(A_{\kbar}))^{\Gamma_k}.$$ Arguing as in the proof of Proposition \ref{imag:p}, the $\Zp$-module $\Hom(B_\kbar[p^\infty],B_\kbar^\vee[p^\infty])^{\Gamma_k}$ is naturally a direct summand of $H^2(A_{\kbar}, \Zp(1))^{\Gamma_k}$. Its preimage in $\NS(A_{\kbar})_\Zp^{\Gamma_k}$ corresponds to the $\Zp$-module $$\Hom(B_\kbar,B^\vee_\kbar)_\Zp^{\Gamma_k}=\Hom(B,B^\vee)_\Zp.$$ This concludes the proof.
\end{proof}

\begin{coro}\label{count:c}
	If $\End(B)=\Z$, then $\Tp(\Br(A_{\kbar}))^{\Gamma_k}\neq 0$.
\end{coro}

\begin{proof}
By the assumption, $\Hom(B,B^\vee)_\Zp$ is a $\Zp$-module of rank $1$. Therefore, by Proposition \ref{counter:p}, it is enough to prove that the rank of $\Hom(B_\kbar[p^\infty],B_\kbar^\vee[p^\infty])^{\Gamma_k}$ is greater than $1$. Since $\End(B)=\Z$, the abelian variety $B$ is not supersingular, so that the $p$-divisible group $B[p^\infty]$ admits at least two slopes. By the Dieudonné--Manin classification, this implies that $B_{\ki}[p^\infty]$ is isogenous to a direct sum $\calG_1\oplus \calG_2$ of non-zero $p$-divisible groups over $\ki$. Since $\End(\calG_1)[\tfrac{1}{p}]\oplus \End(\calG_2)[\tfrac{1}{p}]$ embeds into $\End(B_{\ki}[p^\infty])[\tfrac{1}{p}]\simeq \Hom(B_\kbar[p^\infty],B_\kbar^\vee[p^\infty])[\tfrac{1}{p}]^{\Gamma_k}$, we deduce that $\rk_\Zp(\Hom(B_\kbar[p^\infty],B_\kbar^\vee[p^\infty])^{\Gamma_k})>1$, as we wanted.
\end{proof}


\bibliographystyle{ams-alpha}

\end{document}